\documentclass[11pt,reqno,draft]{amsart}
\usepackage{latexsym,amssymb,mathrsfs}

\newtheorem{defi}{Definition}[section]

\newtheorem{cor}[defi]{Corollary}

\newtheorem{lemma}[defi]{Lemma}
\newtheorem{theorem}[defi]{Theorem}

\newcommand{\defeq}{\mathrel{\mathrm{\raise0.1ex\hbox{:}\hbox{=}\strut}}}
 
 \DeclareMathOperator{\s}{span}
\DeclareMathOperator{\rg}{rg}
\DeclareMathOperator{\Tr}{Tr} \DeclareMathOperator{\Id}{Id}

\def\R{\mathbb R}

\def\E{\mathbb E}

\def\FCb{\mathcal{F}C_b^{\infty}}

\def\Sig{{\mathbf{\sigma}}}
\def\la{\langle}
\def\ra{\rangle}

\def\V{V_{\gamma}}

\def\Lip{\mathop{\rm Lip}}

\def\V{V_{\alpha,\beta}}

\def\uab{u_{\alpha,\beta}}

\def\va{\varphi_{\alpha,\beta}}

\def\norm{\mathcal{N}}

\def\FCb2{\mathcal{F}C_b^{2}}

\def\P{{\mathcal P}}

\begin{document}

\numberwithin{equation}{section}

\title[Estimates for the ergodic measure related to SCSF]{Estimates for the ergodic measure and polynomial stability of Plane Stochastic Curve Shortening  Flow}
\author[A. Es--Sarhir]{Abdelhadi Es--Sarhir}
\author[M-K. von Renesse]{Max-K. von Renesse}
\author[W. Stannat]{Wilhelm Stannat}
\address{Technische Universit\"at Berlin, Institut f\"ur Mathematik \newline Stra{\ss}e des 17. Juni 136, D-10623 Berlin, Germany}
\address{Technische Universit\"at Darmstadt, Fachbereich Mathematik, \newline Schlo\ss gartenstra\ss e 7,
D-64289 Darmstadt, Germany}
\email{stannat@mathematik.tu-darmstadt.de}

\email{[essarhir,mrenesse]@math.tu-berlin.de}
\thanks{The first two authors acknowledge support
from the DFG Forschergruppe 718 "Analysis and Stochastics in Complex
Physical Systems".}

\def\subvstern{_{E^*}}

\keywords{Degenerate stochastic equations, invariant measures,
moment estimates.}


\begin{abstract}
We establish  moment estimates for the invariant measure $\mu$
of  a stochastic partial differential equation describing
motion by mean curvature flow in (1+1) dimension, leading to polynomial stability of the associated Markov semigroup. We also prove   maximal dissipativity on $L^1(\mu)$ for the related Kolmogorov operator.
\end{abstract}

\maketitle
\section{Introduction and preliminaries}
We study the invariant measure $\mu $ on $L^2(0,1)$ and the stability of the following SPDE for a function $u(t)\in L^2(0,1)$ introduced in \cite{Es-Re}, describing curve shortening flow in (1+1)D driven by  additive noise
\begin{equation}
\label{sde0} du(t)= (\arctan u_x(t))_x dt +
  \Sig dW_t,\quad  t\geq 0.  \end{equation}
Here  $W$ is  cylindrical white noise on a separable Hilbert space $U$ defined  on a filtered probability space
$(\Omega,\mathcal{F},(\mathcal{F}_t)_{t\ge 0},\mathbb{P})$ and $ \Sig $ is a Hilbert-Schmidt operator from $U$ to the Sobolev space $H^{1}_0(0,1)$. 
Existence of a unique generalized Markov solution of \eqref{sde0} and its ergodicity  were  shown in \cite{Es-Re}, working in the variational SPDE framework of Pardoux resp.\ Krylov-Rozovski\u\i.  However, certain modifications of standard arguments apply since in contrast to previous works (like e.g.\    \cite{BD}) on  variational  SPDE   the drift operator in \eqref{sde0} is neither coercive nor strongly dissipative. As a consequence  exponential stability of the semigroup cannot be expected here, and it is our main goal to  establish polynomial stability instead (see corollary \ref{regularite Holderienne} below). To this aim we derive moment estimates for  the invariant measure of \eqref{sde0} which become crucial for the control of the contraction by the drift of \eqref{sde0} along the flow. \\
As a second application   we establish the maximal  dissipativity of the Kolmogorov operator $J_0$ associated to \eqref{sde0}, acting on smooth test functions $\varphi: L^2(0,1) \mapsto \R$ by     
\begin{equation} \label{defjo}
J_0\varphi(u)=\frac 12 \Tr  {Q}D^2\varphi(u)+\left \la \frac{
u_{xx}}{1+{u_x^2}},D\varphi(u)\right\ra, \quad u\in D_0,
\end{equation}
with the covariance operator   $Q =\Sig  \Sig^* $  on $L^2(0,1)$ and 
\begin{equation} \label{defd0}
D_0\defeq \bigl  \{u\in W^{1,1}_{loc}(0,1)\,|\, (\arctan (u_x))_x\in
L^2(0,1)\bigr\}.
\end{equation} 
In contrast to the variational approach, here   we shall work with a realization of the drift as a maximally monotone operator on $L^2(0,1)$ given by a subgradient $V=\partial \Phi $ of a convex l.s.c.\ functional $\Phi$ on $L^2(0,1)$, using results of Andreu et al.\ \cite{ACM} for variational PDE of  linear growth functionals. Combining this  with the moment estimates  we prove that 
operator $J_0$ defined on the domain $D(J_0)= C^2_b(H)\subset L^1(H,\mu)$ with $H=L^2(0,1)$ is
closable on $L^1(H,\mu)$ and its closure generates a strongly
continuous Markov semigroup on $L^1(H,\mu)$ (cf.\ \cite{St:99} for related results). 
 
\section{Moment estimates for the invariant measure}

In the sequel we denote by $(e_k)_{k\geq 0}$ the system
of eigenfunctions corresponding to the Laplace operator $\Delta$
on $(0,1)$ with Dirichlet boundary condition. For $n\geq 1$ we denote by
$H_n\defeq \s\{e_1,\cdots, e_n\}$ and $E\defeq H_0^1(0,1)$ and hence
$E^{\ast}=H^{-1}(0,1)$. Recall also that $u\in L^1_{loc}(0,1)$ belongs to the space $BV$ of bounded variation functions if
$$
\bigl[Du\bigr]\defeq\sup\Bigl\{\int_{[0,1]}uv_x\:d\xi:\:\:v\in
C_0^{\infty}(0,1),\:\|v\|_{\infty}\leq 1\Bigr\}<+\infty.
$$
The main result of this section reads as follows.
\begin{theorem} \label{momentthm}
The measure $\mu$ is concentrated on the subset $D_0 \cap \{ u \in L^2 (0,1)\,| \,u_x \in BV(0,1)\}$ and  
$$
\int\bigl[Du_x\bigr] ^{\frac 12}\:\mu(du) + \int \|u\|_{E}^{\frac 1 2 }\:\mu(du)+\int
\| (\arctan u_x)_x\|_{L^2(0,1)}^2\:\mu(du)<+\infty.
$$
\end{theorem}
\begin{proof}
Introducing the  operator $A:\:E\rightarrow E^{\ast}$
$$
\la Au,v\ra=-\int_0^1 \arctan u_x(z)\cdot v_x(z)\:dz,\quad u,\:v\in E.
$$
we may write  \eqref{sde0} as a variational SPDE in the Gelfand triple $E \subset H \subset E^*$  as
\begin{equation*}
 du(t)=A u(t)dt +  \Sig dW_t,\quad t\geq 0.
\end{equation*}

\noindent Below we write ${\subvstern\la} .,. \ra_E$ for the duality  in $E^*\times E$, whereas $\langle .,.\rangle_E$ denotes the inner product in $E$, i.e.\ ${\subvstern\la} \xi,\zeta \ra_E = \langle \xi,  \zeta\rangle_{L^2(0.1)}$ and $ \la \xi,\zeta \ra_E = \langle \xi_x,  \zeta_x\rangle_{L^2(0.1)}$ for $\xi, \zeta \in C^\infty_c(0,1)$.  \\

It is not difficult to see that the operator $A$ satisfies
the following properties.

\begin{itemize}
 \item[(H1)]  For all $u,\:v,\:x \in E$ the map
\[ \R \ni \lambda \to \subvstern\langle A(u+\lambda v ), x\rangle_E\]
is continuous.
\item [(H2)] (Monotonicity)
For all $u$, $v\in E$
\begin{equation*}
\ \subvstern \la Au-Av,u-v\ra_E \leq 0.
\end{equation*}
\item [(H3)]  For  $n \in \mathbb N$, the operator $A$ maps   $H^n:=\mathop{\rm span}\{e_1, \dots, e_n\}\subset E$ into  $E$ and  there  exists a constant
 $c_1 \in \R$  such  that
 \begin{equation*}
 \la Au,u\ra_E + \| \Sig \|_{L_2(U,E)}^2 \leq c_1(1+ \| u \|_E^2 )\quad
\forall u \in H^n, \, n\in \mathbb N.
\end{equation*}
\item [(H4)] There exists a constant $c_2 \in \R$ such that
\[ \|A(u) \|\subvstern \leq c_2(1+\| u\|_E).\]
\end{itemize}

Define $\P_n:\:E^{\ast}\rightarrow H_n$ by
$$
\P_n y\defeq \sum\limits_{i=1}^n\:_{E^{\ast}}\la
y,e_i\ra_{E}e_i,\quad y\in E^{\ast}.
$$
\noindent Then $\P_n|_H$ is just the orthogonal projection onto $H_n$
in $H$. Define the family of $n$-dimensional Brownian
motions in $U$ by  
$$
W^n_t\defeq \sum\limits_{i=1}^n\la W_t,f_i\ra_{U}
f_i=\sum\limits_{i=1}^n B^i(t)f_i,
$$

\noindent where $(f_i)_{i\geq 1}$ is an orthonormal basis of the
Hilbert space $U$. The $n$-dimensional SDE in $H$ 
 \begin{equation}\label{n-sde}
 \left\{
\begin{array}{ll}
du^n(t)=\P_n A
u^n(t)dt+\P_n  \Sig dW^n_t \\
u^n(0,x)=\P_n u_0(x)
\end{array}
\right.
\end{equation}
may be identified with a corresponding SDE $dx(t)= b ^n (x(t)) dt
+  \Sig ^n(x(t))d B^n_t$ in $\R^n$ via the isometric map $\R^n \to H^n, x
\to \sum _{i=1}^n x_i e_i$. By \cite[remark 4.1.2]{Ro} conditions
(H1) and (H2) imply the continuity of the fields $x\to b^n(x)\in
\R^n$. Moreover, assumption (H2) implies
\[
 \langle b^n (x) -b^n(y), x-y\ra_{\R^n}\leq c |x-y|^2, \quad \forall x,y \in \R^n
\]
and, by the equivalence of norms on $\R^n$, (H3) gives the bound
\[  \la b^n(x),x\ra + \| \Sig ^n\|_{L_2(\R^n,\R^n)} \leq c (1+|x|^2),
\]
for some $c>0$. Hence, equation \eqref{n-sde} is a weakly
monotone and coercive equation in $\R^n$ which has a unique globally
defined solution,
 cf.\ \cite[chapter 3]{Ro}. It is proved in \cite{Es-Re} that for  initial datum $u_0\in E$, the
 process $(u^n(t))_{t\geq 0}$ converges $dt$-a.e. in $H$ to a process $(u(t))_{t\geq
 0}$.\\

 As in \cite{Es-Re} we  apply the It\^{o} formula in finite dimensions to derive for $t\to \|u^n(t)\|^2_E$
 \begin{equation*}
\begin{split}
\|u^n(t)\|^2_E&=\|u_0^n\|_E^2+2\int_0^t\la \P_n
A(u^n(s)),u^n(s)\ra_{E}\:ds+\int_0^t\|\P_n \Sig \|^2_{L_2(U_n,E)}\:ds\\&~~~+
M^n(t),\quad t\in[0,T],
\end{split}
\end{equation*}

where \begin{equation*} M^n(t)\defeq 2\int_0^t\la
u^n(s),\P_n \Sig \:dW^n_s\ra_E
\end{equation*}

\noindent and 
$$
\la \P_n A(u^n(s)),u^n(s)\ra_{E}=-\int_{(0,1)}\frac{(u^n_{xx})^2}{1+(u_x^n)^2} dx.
$$
Taking expectation together with   $\|\P_n u_0(x)\|_E\leq \|u_0\|_E$
this entails
\begin{equation}\label{Bound1}
 \frac{1}{t}\E\int_0^t\int_{(0,1)}\frac{(u_{xx}^n(s))^2}{1+(u_x^n(s))^2}dx\:ds<C_1
\end{equation}
for some positive constant $C_1$ independent of $n$ and $t$. On the other hand, the It\^o formula for $\|u^n(t)\|^2_H$ reads 
 \begin{equation}
\begin{split}
\|u^n(t)\|_H^2&=\|u_0^n\|_H^2+2\int_0^t\la \P_n
A(u^n(s)),u^n(s)\ra_H\:ds+\int_0^t\|\P_n \Sig \|^2_{L_2(U_n,H)}\:ds\\&~~~+
N ^n(t),\quad t\in[0,T],
\end{split} \label{itoh}
\end{equation}
with \begin{equation*} N^n(t)\defeq 2\int_0^t\la
u^n(s),\P_n \Sig \:dW^n_s\ra_H
\end{equation*}
\noindent and 
$$
\la \P_n A(u^n(s)),u^n(s)\ra_{H}=\int_{(0,1)} \frac{u^n_{xx}}{1+(u_x^n)^2} u^n \:dx. = - \int_{(0,1)}   u^n_x \cdot \arctan(u^n_{x}) \:dx 
$$
Dividing by $t$ and taking expectation in \eqref{itoh}, using $\arctan s \cdot s \geq |s| - K$ for some $K >0$  yield
\begin{equation}  \frac 1 t \mathbb E \int_0^t \int_{(0,1)}|u^n_x(s)|\:dx \:ds \leq C_2
\label{hnorbd} \end{equation}
for some $C_2>0$. In particular, by the compactness of the embedding $W^{1,1}_0(0,1) \subset L^2(0,1)$ for each $n \in \mathbb N$ the family of measures  $\nu(n,t)(du)\defeq
\frac{1}{t}\int_0^t\mathbb P(u^n(s) \in du)\:ds$, $t \geq 0$, is   tight on $L^2(0,1)$. By analogous arguments as in 
\cite{Es-Re} ergodicity of the Markov semigroup $(P_t^n)_{t \geq 0}$ on $L^2(0,1)$ associated to $(u_t^n)_{t\geq 0}$ holds.  Denoting by $\nu_n$ the corresponding invariant distribution on $L^2(0,1)$, we may thus infer from \eqref{hnorbd}, \eqref{Bound1} and Birkhoff's ergodic theorem that for  arbitrary $L>0$   
\begin{equation*} \label{sqbd} \int  \Bigl(\int_{(0,1)}|u^n_x|\:dx\wedge L  \Bigr) \:\nu_n(du) +  \int
\Bigl(\int_{(0,1)}\frac{{u_{xx}^2}}{1+{u_x^2}}dx \wedge L\Bigr) \:\nu_n(du)<C 
\end{equation*}
where $C=C_1+C_2$. Letting tend $L$ to infinity, by Fatou's lemma we obtain 
 \begin{equation}\label{n-borne} \sup\limits_{n\geq 1} \int \int_{(0,1)}
|u_x|\:dx\:\nu_n(du) + \sup\limits_{n\geq 1} \int \int_{(0,1)}
\frac{{u_{xx}^2}}{1+{u_x^2}}\:dx\:\nu_n(du)<+\infty.
\end{equation}

Since 
\begin{equation*}
   \|(\arctan
u_x)_x\|_{L^2(0,1)}^2 = \int_{(0,1)}
\frac{u_{xx}^2}{(1+u_x^2)^2}dx \leq  \int_{(0,1)}
\frac{u_{xx}^2}{1+u_x^2}dx 
\end{equation*}
this implies 
\begin{equation} \label{estimate-arctan}
\sup\limits_{n\geq 1}\int_H \| (\arctan
u_x )_x\|_{L^2(0,1)}^2\:\nu_n(du)<+\infty.
\end{equation}

\noindent Again, due to the compactness of   $W^{1,1}_0(0,1) \subset H$ the bound \eqref{n-borne} implies that the sequence  $(\nu_n)_{n\geq 1}$ is tight w.r.t.\ the $H$-topology.  This will now be amplified. 

\begin{lemma} \label{tightness in E} For  $u \in C_0^\infty(0,1)$  
\[\left(\int_{(0,1)} |u_{xx}(x)|\:dx\right)^{\frac 12}\leq \frac 12 \int
_{(0,1)} \frac {u_{xx}^2(x)}{1+u_x^2(x)}\:dx +\frac{3}{2} +\frac 12 \|u_x\|_{L^1(0,1)}.
\]
\end{lemma}
\begin{proof}
Starting from  
\begin{equation*}
\int_{(0,1)} |u_{xx}(x)|\:dx\leq \left(\int_{(0,1)} \frac
{u^2_{xx}(x)}{1+u^2_x(x)}\:dx\right)^{\frac 12}\left(\int_{(0,1)}
(1+u^2_x(x))\:dx\right)^{\frac 12},
\end{equation*}

\noindent we get 
\begin{equation*}
\begin{split}
\left(\int_{(0,1)} |u_{xx}(x)|\:dx\right)^{\frac 12}&\leq \left(\int_{(0,1)}
\frac {u^2_{xx}(x)}{1+u^2_x(x)}\:dx\right)^{\frac
14}\left(\int_{(0,1)} (1+u^2_x(x))\:dx\right)^{\frac 14}\\
&\leq \frac 14 \int _{(0,1)}\frac {u^2_{xx}(x)}{1+u^2_x(x)}\:dx+\frac
34 \left(\int_{(0,1)} (1+u^2_x(x))\:dx\right)^{\frac 13}.
\end{split}
\end{equation*}
Combining this with 
\begin{align}
\int_{(0,1)} (u_x(x))^2\:dx=-\int_{(0,1)} u_{xx}(x)u(x)\:dx&\leq \int_{(0,1)}
|u_{xx}(x)|\:dx\cdot
\|u\|_{\infty} \nonumber \\
&\leq \int_{(0,1)} |u_{xx}(x)|\:dx\cdot \|u_x\|_{L^1(0,1)} \label{tembed}
\end{align}
the claim is obtained using Youngs inequality 
\begin{equation*}
\begin{split}
\left(\int_{(0,1)} |u_{xx}(x)|\:dx\right)^{\frac 12}&\leq \frac 14 \int_{(0,1)} \frac {u_{xx}^2(x)}{1+u_x^2(x)}\:dx+\frac{3}{4} +\frac{3}{4} \left(\int_{(0,1)} |u_{xx}(x)|\:dx\right)^{\frac
13}\cdot
\|u_x\|^{\frac 13}_{L^1(0,1)}\\
&\leq \frac 14 \int _{(0,1)} \frac {u_{xx}^2(x)}{1+u_x^2(x)}\:dx
+\frac{3}{4} +\frac 12 \left(\int_H
|u_{xx}(x)|\:dx\right)^{\frac 12}+\frac 14 \|u_x\|_{L^1(0,1)}.
\end{split}
\end{equation*}
\end{proof}

\smallskip 
Combining   \eqref{n-borne} with  Lemma \ref{tightness in E} we obtain a uniform bound 
 
\begin{equation}
\sup_n \int _H \Big(\int_{(0,1)}|u_{xx}(x)|\:dx\Big)^{\frac
12}\:\nu_n(du)<\infty.
\label{unifintbd}
\end{equation}

Due to the  compactness of the embedding $W^{2,1}(0,1)\hookrightarrow E$  
%
this implies that the sequence of measures
$(\nu_n)_{n\geq 1}$ is tight  w.r.t.\ the  $E$-topology. Let $\nu$ be the limit of a converging subsequence.  
%
\noindent From the weak convergence of $\nu_n$ to $\nu$ w.r.t.\ the  $E$-topology  and
the fact that  for $\zeta \in L^2(0,1)$ the function $u\rightarrow \langle \zeta , \arctan u_x\rangle_{L^2(0,1)}^2$ is bounded continuous on $E$
we have

$$
\int_H \la e_k,\arctan u_x\ra^2\:\nu(du)=\lim\limits_{n\to
+\infty}\int_H \la e_k,\arctan u_x\ra^2\:\nu_n(du).
$$

Hence for $m\geq 1$

\begin{equation*}
\begin{split}
\sum\limits_{k=1}^{m}\int_H (\pi k)^2\la e_k,\arctan
u_x\ra^2\:\nu(du)&=\lim\limits_{n\to
+\infty}\sum\limits_{k=1}^{m}\int_H (\pi k)^2\la
e_k,\arctan u_x\ra^2\:\nu_n(du)\\
&\leq \lim\limits_{n\to +\infty}\int _H \sum\limits_{k=1}^{+\infty}
(\pi k)^2\la e_k,\arctan u_x\ra^2\:\nu_n(du)\\
&\leq \lim\limits_{n\to +\infty}\int _H \| (\arctan
u_x )_x\|_{L^2(0,1)}^2\:\nu_n(du)<+\infty,
\end{split}
\end{equation*}
 using  \eqref{estimate-arctan} in the last step. Sending   $m$  to infinity we arrive at 
\begin{equation*}
\int _H \| (\arctan
u_x )_x\|_{L^2(0,1)}^2\:\nu(du)=\sum\limits_{k=1}^{+\infty}\int_H
(\pi k)^2\la e_k,\arctan u_x\ra^2\:\nu(du)<+\infty.
\end{equation*}


Moreover, due to 
the lower semicontinuity of $u \to [Du_x]$ w.r.t.\ to the E-topology  \eqref{unifintbd} yields 
\[\int\bigl[Du_x\bigr] ^{\frac 12}\:\nu(du)<\infty.\]
From this and the boundedness of the embedding $W_0^{2,1}(0,1)$ into $W_0^{1,2}(0,1)$ we finally obtain 
\[\int\|u\|_E^{\frac 1 2 } \:\nu(du)<\infty.\]

\noindent It remains to  show that  the measures $\nu$ and $\mu$ coincide.  Recall that for  $T>0$ and regular initial condition $u_0 \in E$  the sequence of Galerkin approximations  $u^n$ converges to $u$  in the space  $L^2([0,T]\times \Omega, H)$,  c.f. \cite[Chap. 4]{Ro}. Hence, for  all  $t>0$,  $\rho >0$, $x \in E$ and bounded Lipschitz function $\varphi : H \mapsto \R$
\[ P_t^{n,\rho} \varphi(x)  := \frac  1  \rho \int_t^{t+\rho} P_s^n \varphi (x) ds  \longrightarrow  P_t^{\rho} \varphi(x) := \frac  1  \rho \int_t^{t+\rho} P_s \varphi (x) ds.\] 
A straightforward application of   It\^o's formula yields for all $n\in \mathbb N$   
\[  \label{unifcont} |P_t^n \varphi (x) -P_t^n \varphi (y) | \leq \Lip(\varphi) \, \|x-y\|_{H} \quad \forall x,y \in H.\]
Hence  the familiy of functions $(P_t^{n,\rho} \varphi)_{n\geq 0}$ is uniformly continuous on $H$, and for  given  compact subset $K \subset H$   the Arzela-Ascoli theorem guarantees the existence of    a subsequence of   $(P^{n, \rho}_t \varphi)_{n\geq 0}$  converging  uniformly on $K$ to $P^{\rho}_t \varphi$. Moreover, by  \eqref{unifintbd}  and Chebyshev's inequality for 
the collection of compact subsets $K_R = \{ u \in H\, | \|u \|_{E} \leq R\}\subset H$ we find  
\[ \lim_{R\to \infty} \sup_n \nu_n (H\setminus K_R) =0.\]
These two facts  allow to select  a further subsequence, still denoted by $n$,   such that 
\[ \lim_{n} \int_H  P_t^{n,\rho} \varphi(x)  \nu_{n} (dx) =  \int_H P_t^{\rho} \varphi(x)  \nu (dx).\]
Since $\nu_n$ is $P_t^n$-invariant the l.h.s.\ above equals  
\[ \lim_{n} \int_H    \varphi(x)  \nu_n (dx) =   \int_H    \varphi(x)  \nu(dx), \]
i.e.\ $\nu$ is $P^{\rho}_t$-invariant, hence also $P_t$-invariant by letting $\rho$ tend to zero. By the uniqueness of invariant measure for the ergodic semigroup $(P_t)$  we conclude that $\nu=\mu$. 
\end{proof}
\section{Polynomial stability}
%

\begin{theorem} \label{ctheorem}
Let $(u_t)_{t\geq 0}$, $(v_t)_{t\geq 0}$ be two solutions of
\eqref{sde0} with initial condition $u_0$, $v_0\in E$. Then we have
for $\alpha\in (0,1]$
\begin{equation*}
\|u_t-v_t\|_H^{2\alpha}\leq t^{-\alpha}\Big(
3^{\alpha}\Big(1+\frac{1}{t}\int_0^t
\|u_s\|_E^{2\alpha}\:ds+\frac{1}{t}\int_0^t
\|v_s\|_E^{2\alpha}\:ds\Big)\Big) \|u_0-v_0\|_H^{2\alpha}.
\end{equation*}
\end{theorem}
\begin{proof}
For the proof of the theorem we need the following elementary assertion.
\begin{lemma}
For $u$, $v\in E$ we have
\begin{equation} \label{V-accroissement}
\subvstern \la V(u)-V(v),u-v\ra_E\leq
-\frac{1}{\Big(1+\|u\|_E^2+\|v\|_E^2\Big)}\|u-v\|^2_H.
\end{equation}
\end{lemma}

\begin{proof}
Let $\gamma(t)\defeq v+t(u-v)$, $t\in [0,1]$. Then
\begin{align}
\subvstern \la V(u)-V(v),u-v\ra_E&=-\int_{(0,1)} \bigl(\arctan
u_x(r)-\arctan v_x(r)\bigr)
\bigl(u_x(r)-v_x(r)\bigr)\:dr\nonumber 
\\&=-\int_{(0,1)}\int_{(0,1)}
\frac{1}{1+\gamma_x^2(t,r)}(u_x(r)-v_x(r))^2\:dr\:dt, \label{stepone}
\end{align}

Note that for $h\defeq u-v$ we have $h(0)=0$ and hence for all $t\in [0,1]$

\begin{equation*}
\begin{split}
h^2(x)&=\left(\int_0^x h_x(r)\:dr\right)^2\leq
\int_0^x\frac{h_x^2(r)}{1+\gamma^2_x(t,r)}\:dr\cdot\int_0^x
(1+\gamma_x^2(t,r))\:dr\\
& \leq
\int_0^x\frac{h_x^2(r)}{1+\gamma^2_x(t,r)}\:dr\cdot\int_0^x
(1+u_x^2(r)+v_x^2(r))\:dr
\end{split}
\end{equation*}
which in view of \eqref{stepone} yields the claim after integration w.r.t.\ $x$ and $t$.
\end{proof}

\noindent For  $u_0$, $v_0\in E$ let  $(u_t)_{t\geq
0}$, $(v_t)_{t\geq 0}$ be the  strong solutions of \eqref{sde0} starting from   $u_0$ resp.\ $v_0$. Then
\begin{equation*}
\frac{1}{2}\frac{d}{dt}\|u_t-v_t\|_H^2={\subvstern\la}
V(u_t)-V(v_t),u_t-v_t\ra_E\leq
-\frac{\|u_t-v_t\|_H^2}{1+ \|u_t\|_E^2+\|v_t\|_E^2 }.
\end{equation*}

\noindent In particular $t\mapsto
\|u_t-v_t\|_H^2 $ is decreasing and thus
\begin{align*}
 \|u_0-v_0\|_H^2&\geq \|u_t-v_t\|_H^2+\int_0^t
\frac{2\|u_s-v_s\|_H^2}{1+\|u_s\|_E^2+\|v_s\|_E^2}\:ds \nonumber \\
&\geq \|u_t-v_t\|_H^2\left(1+\int_0^t
\frac{2}{1+\|u_s\|_E^2+\|v_s\|_E^2}\:ds\right). 
\end{align*}

\noindent Since for any $\alpha\in (0,1]$ by Jensen's inequality 
\begin{equation*}
\begin{split}
\left(1+t^{\alpha-1}\int_0^t\frac{2^{\alpha}}{\Big(1+\|u_s\|_E^2+\|v_s\|_E^2\Big)^{\alpha}}\:ds\right) &\leq 2^{1-\alpha}
\left(1+\int_0^t\frac{2}{1+\|u_s\|_E^2+\|v_s\|_E^2}\:ds\right)^{\alpha} 
\end{split}
\end{equation*}
this implies 
\begin{equation}
\|u_t-v_t\|_H^{2\alpha}\leq 2^{\alpha-1}
\left(1+t^{\alpha-1}\int_0^t\frac{2^{\alpha}}{\Big(1+\|u_s\|_E^2+\|v_s\|_E^2\Big)^{\alpha}}\:ds\right)^{-1}\|u_0-v_0\|_H^{2\alpha}. \label{cest1}
\end{equation}

\noindent Furthermore, using again Jensen for the convex function $1/x$  
\begin{equation*}
\begin{split}
\int_0^t\frac{1}{\Big(1+\|u_s\|_E^2+\|v_s\|_E^2\Big)^{\alpha}}\:ds&\geq
\frac{t^2}{\int_0^t\Big(1+\|u_s\|_E^2+\|v_s\|_E^2\Big)^{\alpha}\:ds}\\
&\geq \frac{t^2}{3^{\alpha-1}\Big(t +\int_0^t
\|u_s\|_E^{2\alpha}\:ds+\int_0^t \|v_s\|_E^{2\alpha}\:ds\Big)}\\
&= \frac{t}{3^{\alpha-1}\Big(1 +\frac{1}{t}\int_0^t
\|u_s\|_E^{2\alpha}\:ds+\frac{1}{t}\int_0^t
\|v_s\|_E^{2\alpha}\:ds\Big)},
\end{split}
\end{equation*}
\noindent which inserted into \eqref{cest1} gives
%
%
\begin{equation*}
\begin{split}
\|u_t-v_t\|_H^{2\alpha}&\leq
2^{\alpha-1}\left(1+\frac{2^{\alpha}t^{\alpha}}{3^{\alpha-1}\Big(1 +\frac{1}{t}\int_0^t
\|u_s\|_E^{2\alpha}\:ds+\frac{1}{t}\int_0^t
\|v_s\|_E^{2\alpha}\:ds\Big)}\right)^{-1}\|u_0-v_0\|_H^{2\alpha}\\
&\leq
2^{\alpha}\left(1+\frac{2^{\alpha}t^{\alpha}}{3^{\alpha}\Big(1 +\frac{1}{t}\int_0^t
\|u_s\|_E^{2\alpha}\:ds+\frac{1}{t}\int_0^t
\|v_s\|_E^{2\alpha}\:ds\Big)}\right)^{-1}\|u_0-v_0\|_H^{2\alpha}\\
&= 2^{\alpha}\frac{3^{\alpha}\Big(1 +\frac{1}{t}\int_0^t
\|u_s\|_E^{2\alpha}\:ds+\frac{1}{t}\int_0^t
\|v_s\|_E^{2\alpha}\:ds\Big)}{2^{\alpha}t^{\alpha}+3^{\alpha}\Big(1 +\frac{1}{t}\int_0^t
\|u_s\|_E^{2\alpha}\:ds+\frac{1}{t}\int_0^t
\|v_s\|_E^{2\alpha}\:ds\Big)}\|u_0-v_0\|_H^{2\alpha}\\
&\leq t^{-\alpha}\Big( 3^{\alpha}\Big(1 +\frac{1}{t}\int_0^t
\|u_s\|_E^{2\alpha}\:ds+\frac{1}{t}\int_0^t
\|v_s\|_E^{2\alpha}\:ds\Big)\Big) \|u_0-v_0\|_H^{2\alpha}.
\end{split}
\end{equation*}
\end{proof}

\noindent Choosing $\alpha=\frac 14$, we obtain in particular
\begin{equation}\label{1/4-Hoelder}
\|u_t-v_t\|_H^{\frac 12}\leq t^{-\frac 14}C\Big(1+\frac{1}{t}\int_0^t
\|u_s\|_E^{\frac 12}\:ds+\frac{1}{t}\int_0^t \|v_s\|_E^{\frac
12}\:ds\Big)\Big) \|u_0-v_0\|_H^{\frac 12},
\end{equation} 
for some positive constant $C$. As a consequence we arrive at the following statement.
\begin{cor}\label{regularite Holderienne}
Let  $\varphi: L^2(0,1) \mapsto \R$ bounded and $\frac 1 2 $-H\"older-continuous, i.e. 
\[ \sup_{x\ne y \in L^2(0,1)} \frac {\varphi(x) - \varphi(y)}{\|x-y\|_{L^2(0,1)}^{1/2}} =:  \bigl|\varphi\bigr|_{1/2} <\infty,\] 
then  for $u,v \in E$
\[  \limsup_{t \to \infty} 
\left[  t^{1/4} \frac {|P_t\varphi (u) - P_t\varphi (v) | }{\|u-v\|_{L^2(0,1)}^{1/2}} 
\right] \leq \frac 1 {\sqrt 2}  \bigl|\varphi\bigr|_{1/2} C  \left( 1+2 \int \|u\|_E^{\frac 1 2 }\,\mu(du) \right).\]
\end{cor}

\begin{proof}
Using \eqref{1/4-Hoelder} \begin{align*}
|P_t\varphi(u)-P_t\varphi(v)|& = | \E(\varphi(u_t)-\varphi(v_t))|\leq
 \bigl|\varphi\bigr|_{1/2}    \E(\|u_t-v_t\|_H^{\frac 1 2 })\\
&\leq  t^{-\frac 1 4} \|u-v\|_H^{\frac 1 2 } \cdot
 C\bigl(1+\E  (\frac{1}{t}\int_0^t \|u_s\|_E^{\frac
12}\:ds+\frac{1}{t}\int_0^t \|v_s\|_E^{\frac 12}\:ds )\Bigr) ,
\end{align*}
where  $ \frac{1}{t} \E \int_0^t \|u_s\|_E^{\frac 1 2 }\:ds$
converges to $\displaystyle\int_H \|u\|^{\frac 1 2 }_E\:\mu(du)$ as
$t\to+\infty$, due to the ergodicity of $(P_t)$. \end{proof}
\section{Maximal dissipativity of the operator $J$}

In this final  section we prove the maximal dissipativity of  the operator $(J_0,D(J_0))$
 on the space $L^1(H,\mu)$, where $J_0$ is  defined in \eqref{defd0} and $D(J_0):= C^2_b(H)$. As a standard consequence the transition semigroup $(P_t)$ corresponding to the generalized solution of \eqref{sde0} admits a unique extension to a strongly continuous semigroup $(P^0_t)_{t\geq 0}$ on
$L^1(H,\mu)$. \smallskip

For the proof  we exploit  that the drift in \eqref{sde0} can be associated to a subdifferential of  a
convex l.s.c.\ functional on $L^2(0,1)$, using  the general set-up  introduced in  \cite{ACM} for $L^2$-gradient flows of linear growth functionals.  Let $G$  denote  the primitive of the function
$s\mapsto\arctan s$, then  $G$ is a convex function with linear
growth at infinity. For a measure $\nu$ on $[0,1]$ with  Lebesgue
decomposition
$$
\nu\defeq h dx+\nu^s
$$
with   $\nu=h|\nu|$ and $\nu^s$ is singular part of $\nu$, we define a new measure $G(\nu)$ 
on the  Borel sets $B\subset [0,1]$  by 
 $$
\int_B G(\nu)\defeq \int_B G(h(x))\:dx+\int_B
G_{\infty}\Big(\frac{d\nu}{d|\nu|}\Big)\:d|\nu|^s.
$$
where 
$$
G_{\infty}(x)\defeq \lim\limits_{t\to+\infty}\frac{G(tx)}{t}= \frac \pi 2  x. 
$$

We introduce the   functional $\Phi$ on $L^2(0,1)$ 
\[\Phi(x)=\left\{\begin{array}{ll} \int_{[0,1]} G(Du),&\quad u\in BV(0,1)\\
+\infty,&\quad u\in L^2(0,1)\setminus BV(0,1). 
\end{array}
\right.
\]
\noindent By the results in  \cite{ACM} the functional  
$\Phi$ is convex on $BV(0,1)$ and lower semicontinuous on
every $L^p(0,1)$. Hence the
subdifferential $\partial \Phi$ of $\Phi$, which is the multi-valued operator in
$L^2(0,1)$ defined by
$$
v\in\partial \Phi\:\:\Longleftrightarrow\:\: \Phi(\zeta)-\Phi(u)\geq
\int_{(0,1)} v(\zeta-u)\:dx,\:\:\forall\:\zeta\in L^2(0,1)
$$
is a maximal monotone operator in $L^2(0,1)$. Clearly $u\in
BV_0^1(0,1)$ if $u\in W^{1,1}_0(0,1)$ with $\|Du\|=\| u_x\|_{L^1(0,1)}$ and so if
$$
v=- (\arctan u_x)_x\in L^2(0,1),
$$
then $v\in \partial \Phi(u)$. Moreover, since 
$$
\mbox{For $\zeta\in \R$}, \quad |\zeta|-C_1\leq \arctan
\zeta\cdot\zeta\quad \quad \mbox{for some }\quad C_1>0
$$
  $u\in D_0$ implies  $u\in
W^{1,1}(0,1)$, i.e.\ with   $V(u)\defeq -\partial\Phi(u)$ for $u\in D_0$ we find  that $V(u)= {u_{xx}}/({1+{u_x^2}})$.
Thus, considering  the Kolmogorov operator $J_0$ as unbounded
operator in $L^1(H,\mu)$ with domain $D(J_0)\defeq C_b^2(H) $ we
can write for $\varphi \in C_b^2(H)$
$$
J_0\varphi(u)=\frac 12 \Tr  {Q}D^2\varphi(u)+\left \la
V(u),D\varphi(u)\right\ra, \quad \varphi\in C_b^2(H).
$$

\noindent Note that this definition of $J_0$ makes sense in
$L^1(H,\mu)$, because by  Theorem \ref{momentthm} 
\[\int_H V(u)^2\:\mu(du)<+\infty.\]
Secondly, it follows from   It\^o's-formula for $\|u(t)\|_H^2$ for solutions with regular initial condition that  the measure $\mu$ is
infinitesimally invariant for the operator $J_0$, i.e.\ 
$$
\int J_0\varphi(u)\:\mu(du)=0\quad\mbox{for all $\varphi\in
D(J_0)$},
$$
and moreover, since 
$$
J_0\varphi^2=2\varphi J_0\varphi+\frac 12 \la
Q  D\varphi,  D\varphi\ra,\quad \varphi\in D(J_0)
$$ 
also
\[
\int_H J\varphi(u) \varphi(u)\:\mu(du)=-\frac 12 \int _H \| \Sig ^* D\varphi(u)\|^2\:\mu(du).
\] which entails that  $J_0$ is
dissipative in the Hilbert space $L^2(H,\mu)$. By   similar
argument as in \cite{Es-St} one  proves that $J_0$ is also dissipative
in $L^1(H,\mu)$. Therefore it is closable and its closure $J\defeq
\bar{J_0}$ with domain $ D(J)$ is dissipative. Now   the  main assertion of this 
 section reads as follows.
\begin{theorem}
The operator $(J,D(J))$  generates a $C_0$-semigroup of contractions on $L^1(H,\mu)$.
\end{theorem}
\begin{proof}
We shall prove that $\rg(\lambda -J)$ is dense in $L^1(H,\mu)$. To
this aim for $\alpha>0$ consider the
Yosida approximation of $V$ defined by
\begin{equation*}
V_\alpha(x)=V(J_{\alpha}(x)),\quad\mbox{where}\quad
J_{\alpha}(x)=(\Id-\alpha V)^{-1}(x),\quad x\in D(V).
\end{equation*}
\noindent For the sequence $V_\alpha$ we have the following:
\begin{enumerate}
\item[(i)] For any $\alpha>0$, $V_{\alpha}$ is dissipative and Lipschitz continuous.
\item [(ii)] $\|V_{\alpha}(x)\|\leq \|V(x)\|$ for any $x\in D(V)$.
\end{enumerate}

\noindent Note that the function $V_{\alpha}$ is not differentiable
in general. Therefore we shall consider a $C^1$-approximation as in
\cite{DaZ3}. For $\alpha$, $\beta>0$ we set
$$
V_{\alpha,\beta}(x)\defeq \int_H e^{\beta \Delta}V_{\alpha}(e^{\beta
\Delta}x+y)\norm_{0, \Sig _{\beta}}(dy)
$$
where $\norm_{0, \Sig _{\beta}}$ is the Gaussian measure on $H$ with mean
$0$ and covariance operator defined by
$ \Sig _{\beta}\defeq\int_0^{\beta}e^{2s\Delta}\:ds$. Then, $\V$ is
dissipative and by the Cameron-Martin formula it is $C^{\infty}$
differentiable. Moreover, as $\alpha$, $\beta\to 0$, $\V\to V$
pointwise. Let us now introduce the following approximating equation
\begin{equation}
\label{uab} \left\{
\begin{array}{ll}
d\uab(t)=\V(\uab(t))dt+ \Sig dW_t,\quad
t\geq0\\
\uab(0)=x.
\end{array}
\right.
\end{equation}

Since $\V$ is globally Lipschitz, equation \eqref{uab} has a unique
strong solution $(\uab(t))_{t\geq 0}$. Moreover by the regularity of
$\V$ the process $(\uab(t))_{t\geq 0}$ is differentiable on $H$. For
any $h\in H$ we set $\eta_h(t,x)\defeq D\uab(t,x)\cdot h $ it holds 
\begin{equation}\label{eqet}
 \left\{
\begin{array}{ll}
\frac{d}{dt}\eta_h(t,x)=D\V(\uab(t,x))\cdot\eta_h(t,x),\\
\eta^h(0,x)=h\in H.
\end{array}
\right.
\end{equation}
\noindent From the dissipativity of $\V$ we have that
$$
\la D\V(z)h,h\ra\leq 0,\quad h\in H, z\in D(V).
$$

 \noindent Hence by multiplying
both sides of \eqref{eqet} by $\eta_h(t,x)$, integrating with
respect to $t$, we have
\begin{equation}\label{bound-eta}
\|\eta_h(t,x)\|^2\leq \|h\|^2.
\end{equation}

 Now for $\lambda>0$ and $f\in C_b^2(H)$, consider the following elliptic equation
\begin{equation}\label{elliptic}
(\lambda-J_{\V})\va=f,\quad \lambda>0.
\end{equation}

where $J_{\V}$ is the Kolmogorov operator corresponding to the SDE
\eqref{uab}.
\noindent It is well-known that this equation has a solution $\va\in
C_b^2(H)$ and can be written in the form $\va=R(\lambda,J_{\V})f$,
where
$$
\bigl(R(\lambda,J_{\V})f\bigr)(x)=\int_0^{+\infty}e^{-\lambda t}\E(f(\uab(t,x)))\:dt
$$
is the pseudo resolvent associated with $J_{\V}$. Thus we have
\begin{equation}\label{res-estimate}
\|\lambda\va\|_\infty\leq\|f\|_\infty.
\end{equation}
\noindent We have, moreover, for all $h\in H$,
\begin{equation*}
D\va(x)h=\int_0^{+\infty}e^{-\lambda
t}\E\Big(Df(\uab(t,x))(D\uab(t,x)h)\Big)\:dt.
\end{equation*}
\noindent consequently by using \eqref{bound-eta} it follows that
\[\sup\limits_{\alpha,\beta>0}\|D\va(x)\|\leq
\frac{1}{\lambda}\|Df\|_{\infty}.\]
\noindent From \eqref{elliptic} we have
\[\begin{split}
\lambda\: \va(x)&-\frac 12 \Tr  {Q}D^2\varphi(x)+\la V(x),D\va(x)\ra\\
&=f(x)+\la V(x)-\V(x),D\va(x)\ra,\quad \lambda>0,\:\:x\in D(V).
\end{split}\]

\noindent Using gradient bound  \eqref{res-estimate} we deduce that
\begin{equation*}\begin{split}
\int_H|\la\V(x)-V(x),D\va(x)\ra| \:\mu(dx)\leq
\frac{1}{\lambda^2}\|Df\|^2_{\infty}\|\V-V\|_{L^2 (H, \mu )}.
\end{split}
\end{equation*}
By  Lebesgue's theorem  
$\|\V-V\|_{L^2 (H, \mu )}$ converges to $0$ as $\alpha,\:\beta\to
0$. Therefore we deduce that for $\alpha,\:\beta\to 0$,
$$
\lambda\: \va(x)-\frac 12 \Tr  {Q}D^2\va (x)+\la V(x),D\va(x)\ra\to
f
$$
strongly in $L^1(H,\mu)$. This implies that
$$
C_b^2(H)\subset \overline{(\lambda-J_0)(D(J_0))}.
$$
Since $C_b^2(H)$ is dense in $L^1(H,\mu)$, the proof is complete.
\end{proof}

\end{document}